  \documentclass[10pt]{article}
  \usepackage{amsmath,amsthm,amsfonts,amssymb,latexsym,graphicx} %
  \usepackage[noadjust]{cite}  %
  \usepackage{algorithm}
  \usepackage{algpseudocode}  %
  \usepackage[pdfpagemode=UseNone,pdfstartview=FitH]{hyperref}
  \newcommand{\Extra}[1]{}

\usepackage{stmaryrd} %
\usepackage{tikz} %
\usepackage{forest}
\usepackage{algorithm}
\usepackage{algpseudocode}  %
\usepackage[pdfpagemode=UseNone,pdfstartview=FitH]{hyperref}

\allowdisplaybreaks[4]

  \newtheorem{theorem}{Theorem}
  \newtheorem{lemma}[theorem]{Lemma}
  \newtheorem{corollary}[theorem]{Corollary}
  \newtheorem{proposition}[theorem]{Proposition}
  \theoremstyle{definition}

  \newtheorem{remark}[theorem]{Remark}

\renewcommand{\P}{\mathbb{P}}

\newcommand{\KL}{\textrm{KL}}

\newcommand{\R}{\mathbb{R}}  %

\DeclareMathOperator{\lb}{lb}  %
\DeclareMathOperator{\Do}{do}  %

\usepackage{accents}
\newlength{\dhatheight}
\newcommand{\dhat}[1]{%
  \settoheight{\dhatheight}{\ensuremath{\hat{#1}}}%
  \addtolength{\dhatheight}{-0.35ex}%
  \hat{\vphantom{\rule{1pt}{\dhatheight}}%
  \smash{\hat{#1}}}}

  \title{Confidence intervals for causal effects in sequential decision making}
  \author{Vladimir Vovk and Ruodu Wang}

\begin{document}
\maketitle

\begin{abstract}
  We derive confidence intervals and confidence sequences
  for causal effects in situations where the back-door criterion
  is applicable.
  Our tightest confidence intervals hold in the standard setting
  where the training data consists of IID observations
  over a system described by a given causal diagram.
  When interventions are allowed to depend on the past data,
  our confidence intervals become wider and involve a term
  coming from the law of the iterated logarithm,
  even where the number of observations is known in advance.
  In the sequential setting where the number of observations is not given,
  our confidence intervals, arranged into a confidence sequence for causal effects,
  involve more iterated logarithm terms and become even wider.

     The version of this paper at \url{http://gtfp.net} (Working Paper 68)
     is updated most often.
\end{abstract}

\section{Introduction}

A major limitation of many results
in causal inference is that they assume, implicitly or explicitly,
IID (independent and identically distributed) observations over a causal system.
This limitation is shared by \cite{Vovk/Wang:arXiv2603full},
where we derive prediction sets in situations
covered by the back-door and front-door criteria
\cite[Theorems 3.3.2 and 3.3.4]{Pearl:2009}
(one section of \cite{Vovk/Wang:arXiv2603full} goes beyond the IID picture
but only slightly).
It was noted in \cite{Vovk:1996Gam} that the limitation
can be easily overcome by applying
limit theorems of probability theory,
but the results there (such as \cite[Theorem 1]{Vovk:1996Gam})
are asymptotic.
In this paper we derive finite-sample confidence intervals
in natural non-IID settings.

We start in Sect.~\ref{sec:effect}
by giving the definition of the causal effect
adapted to the back-door and front-door criteria.
In Sect.~\ref{sec:IID} we consider the simplest IID setting
giving the narrowest confidence intervals.
Limitations of this setting are pointed out in, e.g.,
\cite[Sect.~1]{Vovk:1996Gam} and \cite[end of Sect.~3.6.1]{Pearl:2009}.

In Sect.~\ref{sec:core} we drop the assumption of IID data,
which leads to the appearance of an iterated logarithm term in our confidence intervals.
This is developed further in Sect.~\ref{sec:anytime};
there we make our confidence intervals anytime-valid,
which leads to further iterated logarithm terms.
In Sects.~\ref{sec:core}--\ref{sec:anytime} we only consider the back-door criterion.

This paper was motivated by the difficulty of applying
the methods developed in \cite{Vovk/Wang:arXiv2603full}
to the most natural setting of sequential causal inference
(which we called the ``strong interpretation'' of causal diagrams).
The methods of this paper are completely different,
and we are targeting confidence intervals rather than prediction sets,
which were targeted in \cite{Vovk/Wang:arXiv2603full}.
In Sect.~\ref{sec:prediction} we briefly discuss derivation of prediction sets
from our confidence intervals,
but this is likely to lead to much more conservative prediction sets
in the IID setting as compared with \cite{Vovk/Wang:arXiv2603full}.

Finally, Sect.~\ref{sec:conclusion} concludes and lists some directions of further research.

\section{Causal effect}
\label{sec:effect}

\begin{figure}
  \begin{center}
    \begin{tikzpicture} %
      \node (X) at (0,0) {X};
      \node (Y) at (4,0) {Y};
      \node (Z) at (2,1) {Z};
      \draw [->] (Z) edge (X);
      \draw [->] (Z) edge (Y);
      \draw [->] (X) edge (Y);
    \end{tikzpicture}
  \end{center}
  \caption{The basic causal graph of this paper}
  \label{fig:back}
\end{figure}
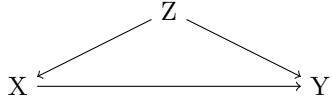

Our running example will be the causal diagram in Figure~\ref{fig:back},
which we now use to explain our notation
(in which we follow mainly Pearl \cite{Pearl:2009}).
The variables, such as $X$, $Y$, and $Z$,
in our causal diagrams will always range over finite sets
denoted by the corresponding boldface letters,
such as $\mathbf{X}$, $\mathbf{Y}$, and $\mathbf{Z}$,
and called their \emph{domains}
(equipped with the discrete $\sigma$-algebras).
However, when talking specifically about the example in Figure~\ref{fig:back},
we will usually assume that $\mathbf{X}=\mathbf{Y}=\mathbf{Z}=\{0,1\}$,
so that $X,Y,Z$ are the indicator functions of events.

Let $P$ be a positive probability measure on $\mathbf{X}\times\mathbf{Y}\times\mathbf{Z}$
(generating the random variables $X$, $Y$, and $Z$).
Suppose it factorizes according to Figure~\ref{fig:back}:
\begin{multline}\label{eq:factor}
  P(X=x,Y=y,Z=z)\\
  =
  P(Z=z)
  P(X=x\mid Z=z)
  P(Y=y\mid X=x,Z=z)
\end{multline}
for all $x\in\mathbf{X}$, $y\in\mathbf{Y}$, and $z\in\mathbf{Z}$.
It will be very convenient to use Pearl's
\cite[Sects.~1.1.4 and~1.1.5]{Pearl:2009}
convention and abbreviate
$P(X=x)$ to $P(x)$, $P(Y=y)$ to $P(y)$, etc.;
such abbreviated notation will also be used
when we have, say, $\tilde x$ or $x'$ in place of $x$.
We will also often omit mentioning that $x\in\mathbf{X}$, $y\in\mathbf{Y}$, etc.
With this convention, we can rewrite \eqref{eq:factor} as
\[
  P(x,y,z)
  =
  P(z)
  P(x\mid z)
  P(y\mid x,z).
\]

Let $\tilde x\in\mathbf{X}$.
We use Pearl's \cite{Pearl:2009} notation $\Do(X=\tilde x)$,
usually abbreviated to $\Do(\tilde x)$,
to signify setting $X$ to $\tilde x$
(we will define what this means formally
only in specific contexts).
Let us define,
in the context of Figure~\ref{fig:back},
the \emph{causal effect} of $X$ on $Y$ as
\begin{equation}\label{eq:back}
  P(y\mid\Do(\tilde x))
  :=
  \sum_{z}
  P(y\mid\tilde x,z)
  P(z).
\end{equation}
(The definition in \cite[Sect.~3.2]{Pearl:2009} is more general,
but it is not our focus in this paper
and we will use simpler \emph{ad hoc} definitions.)
The interpretation of \eqref{eq:back}
(and of causal effects in general)
is that it is the probability of $Y=y$
in the mutilated causal model
in which the arrow from $Z$ to $X$ in Figure~\ref{fig:back} has been removed
and $X$ has been set to $\tilde x$.

The decomposition \eqref{eq:factor} and these notational conventions
generalize to any directed acyclic graph (dag),
and Figure~\ref{fig:back} can be generalized
to the following \emph{back-door criterion},
which is stated in terms of ``blocking'',
as defined in \cite[Definition 1.2.3]{Pearl:2009}.
If $X$, $Y$, and $Z$ are disjoint non-empty sets of variables in a dag,
$Z$ is said to satisfy the back-door criterion relative to $(X,Y)$
if, for any $X'\in X$ and $Y'\in Y$,
\begin{itemize}
\item
  no vertex in $Z$ is a descendant of $X'$,
  and 
\item
  $Z$ blocks every \emph{back-door path} from $X'$ to $Y'$,
  i.e., every path between $X'$ and $Y'$ that contains an arrow into $X'$
\end{itemize}
\cite[Definition 3.3.1]{Pearl:2009}.
If the back-door criterion is satisfied,
the \emph{causal effect} can still be defined as \eqref{eq:back}
\cite[Theorem 3.3.2]{Pearl:2009}.
However, now the summing $\sum_z$ over $z$ in \eqref{eq:back} means
summing over all possible values of the variables in $Z$:
\[
  \sum_z
  :=
  \sum_{z_1\in\mathbf{Z}^1}
  \dots
  \sum_{z_k\in\mathbf{Z}^k},
\]
where $\mathbf{Z}^i$ is the domain of $Z^i$, $i=1,\dots,k$,
and $Z^1,\dots,Z^k$ are the elements of $Z$, $Z=\{Z^1,\dots,Z^k\}$.
Moreover, $\tilde x$ should specify the values for all variables in $X$.
In the theorems below we will use the notation
\begin{equation}\label{eq:Z}
  \left|\mathbf{Z}\right|
  :=
  \left|\mathbf{Z}^1\right|
  \dots
  \left|\mathbf{Z}^k\right|.
\end{equation}

Another set of conditions that allows a natural
(albeit not obvious)
definition of causal effects
is known as the \emph{front-door criterion}.
Now let $X$ and $Y$ be variables and $Z$ be a non-empty set of variables
not containing $X$ and $Y$.
Then we say that $Z$ satisfies the front-door criterion relative to $(X,Y)$
if
\begin{itemize}
\item
  $Z$ intercepts all directed paths from $X$ to $Y$,
  and
\item
  there is no unblocked back-door path from $X$ to $Z$,
  and
\item
  all back-door paths from $Z$ to $Y$ are blocked by $X$
\end{itemize}
\cite[Definition 3.3.3]{Pearl:2009}.
In this case the \emph{causal effect} is defined by
\begin{equation}\label{eq:front}
  P(y\mid\Do(\tilde x))
  :=
  \sum_z
  P(z\mid\tilde x)
  \sum_x
  P(y\mid x,z)
  P(x)
\end{equation}
\cite[Theorem 3.3.4]{Pearl:2009}.

It is very important (but does not concern us in this paper)
that some of the variables in the causal dag may be unobservable;
it's fine as long as these variables do not enter expressions for causal effects,
such as \eqref{eq:back} or~\eqref{eq:front}.

\section{The IID setting}
\label{sec:IID}

We start from the most standard setting
where the observations before intervention are IID;
see, e.g., \cite[the beginning of Sect.~2.2]{Pearl:2009}.
Informally, Nature possesses stable causal mechanisms
that are organized in the form of a graphical structure.
In causal calculus,
the structure is a known dag,
and the stable causal mechanisms are unknown probability distributions
of the variables in the vertices of the dag given their parents.
The available observations are generated in the IID fashion.
(The IID nature of the observations usually stays implicit,
and it appears that in Pearl's book \cite{Pearl:2009}
it is made explicitly only in \cite[the end of Sect.~3.6.1]{Pearl:2009}.)

\begin{figure}
  \centerline{%
    \begin{tikzpicture} %
      \node (X1) at (0,0) {$X_1$};
      \node (Y1) at (2,0) {$Y_1$};
      \node (Z1) at (1,1) {$Z_1$};
      \draw [->] (Z1) edge (X1);
      \draw [->] (X1) edge (Y1);
      \draw [->] (Z1) edge (Y1);
      \node (X2) at (4,0) {$X_2$};
      \node (Y2) at (6,0) {$Y_2$};
      \node (Z2) at (5,1) {$Z_2$};
      \draw [->] (Z2) edge (X2);
      \draw [->] (X2) edge (Y2);
      \draw [->] (Z2) edge (Y2);
      \node (XN) at (10.3,0) {$X_N$};
      \node (YN) at (12.3,0) {$Y_N$};
      \node (ZN) at (11.3,1) {$Z_N$};
      \draw [->] (ZN) edge (XN);
      \draw [->] (XN) edge (YN);
      \draw [->] (ZN) edge (YN);
      \node[draw, rectangle] (past1) at (2.5,0.6) {$\text{past}_1$};
      \draw [->] (past1) edge (X2);
      \node[draw, rectangle] (past2) at (6.5,0.6) {$\text{past}_2$};
      \node at (7.4,0.5) {$\dots$}; %
      \node[draw, rectangle] (past3) at (8.5,0.6) {$\text{past}_{N-1}$};
      \draw [->] (past3) edge (XN);
    \end{tikzpicture}}
  \caption{The repeated causal graph}
  \label{fig:back-repeated}
\end{figure}
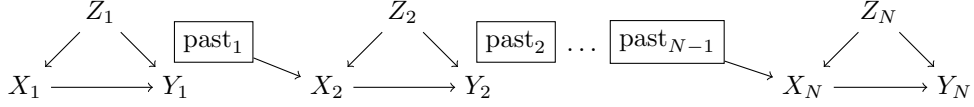

Figure~\ref{fig:back-repeated} shows $N$ repetitions
of the causal system represented in Figure~\ref{fig:back}.
Let us ignore the cells labelled $\text{past}_i$, $i\in\{1,\dots,N-1\}$,
for now (formally, we are assuming that these variables take a fixed known value).
This composite causal diagram then represents $N$ IID observations
over the base diagram of Figure~\ref{fig:back}.
In this section we are interested in Figure~\ref{fig:back-repeated}
with an arbitrary dag as the base diagram,
not necessarily the one in Figure~\ref{fig:back}.
This IID picture will give the narrowest confidence intervals
out of those derived in this paper.

The following theorem treats the case of the back-door criterion,
and it is proved (as all other results in Sects.~\ref{sec:IID}--\ref{sec:anytime})
in Appendix~\ref{app:proofs-1}.
A confidence interval
$[m-h,m+h]$ will be represented
in terms of its \emph{midpoint} $m$ and \emph{half-width} $h$.
The number $N$ of repetitions is fixed, and we set
\begin{equation}\label{eq:number}
  \# x y := \left|\{n\in[N]:(X_n,Y_n)=(x,y)\}\right|,
\end{equation}
where $[N]:=\{1,\dots,N\}$;
i.e., $\# x y$ is the number of times $(X_n,Y_n)=(x,y)$).
The analogous notation will be used for sequences other than $x y$,
such as $x y z$ and $z$.
See \eqref{eq:Z} for the definition of $\left|\mathbf{Z}\right|$.

\begin{theorem}\label{thm:IID-back}
  Suppose the back-door criterion is satisfied.
  Let $\delta>0$, $\tilde x\in\mathbf{X}$, and $y\in\mathbf{Y}$.
  The following is a $(1-\delta)$-confidence interval
  for the parameter $P(y\mid\Do(\tilde x))$ defined by \eqref{eq:back}:
  the midpoint is
  \begin{equation}\label{eq:IID-back-1}
    \sum_z
    \hat p(y\mid\tilde x,z)
    \hat p(z)
  \end{equation}
  and the half-width is
  \begin{equation}\label{eq:IID-back-2}
    \left|\mathbf{Z}\right|
    \sqrt{\frac{\ln\frac{4\left|\mathbf{Z}\right|}{\delta}}{2N}}
    +
    \sum_z
    \sqrt{\frac{\ln\frac{4\left|\mathbf{Z}\right|}{\delta}}{2\#\tilde x z}},
  \end{equation}
  where $\hat p(y\mid\tilde x,z):=\#\tilde x y z / \#\tilde x z$
  is the standard estimate for $P(y\mid\tilde x,z)$
  and $\hat p(z):=\#z / N$ is the standard estimate for $P(z)$.
  (The half-width~\eqref{eq:IID-back-2} is understood to be $\infty$
  when $\#\tilde x z=0$ for some $z$.)
\end{theorem}

In the case of Figure~\ref{fig:back} with binary variables,
we can replace \eqref{eq:IID-back-2} by
\begin{equation}\label{eq:IID-back-2-toy}
  2\sqrt{\frac{\ln\frac{6}{\delta}}{2N}}
  +
  \sum_{z\in\{0,1\}}
  \sqrt{\frac{\ln\frac{6}{\delta}}{2\#\tilde x z}}
\end{equation}
(although this does not quite follow from \eqref{eq:IID-back-2}).

The following is the analogue of Theorem~\ref{thm:IID-back}
for the front-door criterion.

\begin{theorem}\label{thm:IID-front}
  Suppose the front-door criterion is satisfied.
  Fix $\delta>0$, $\tilde x$, and $y$.
  The following is a $(1-\delta)$-confidence interval
  for the parameter $P(y\mid\Do(\tilde x))$ defined by \eqref{eq:front}:
  the midpoint is
  \begin{equation}\label{eq:IID-front-1}
    \sum_z
    \hat p(z\mid\tilde x)
    \sum_x
    \hat p(y\mid x,z)
    \hat p(x)
  \end{equation}
  and the half-width is
  \begin{equation}\label{eq:IID-front-2}
    \left|\mathbf{X}\right|\left|\mathbf{Z}\right|
    \sqrt{\frac{\ln\frac{2K}{\delta}}{2N}}
    +
    \left|\mathbf{Z}\right|
    \sqrt{\frac{\ln\frac{2K}{\delta}}{2\#\tilde x}}
    +
    \sum_{x,z}
    \sqrt{\frac{\ln\frac{2K}{\delta}}{2\#x z}},
  \end{equation}
  where
  \begin{equation*} %
    K
    :=
    \left|\mathbf{X}\right|\left|\mathbf{Z}\right|+\left|\mathbf{X}\right|+\left|\mathbf{Z}\right|
    =
    (\left|\mathbf{X}\right|+1)(\left|\mathbf{Z}\right|+1) - 1,
  \end{equation*}
  $\hat p(z\mid\tilde x):=\#\tilde x z / \#\tilde x$
  is the standard estimate for $P(z\mid\tilde x)$,
  $\hat p(y\mid x,z):=\#x y z / \#x z$
  is the standard estimate for $P(y\mid x,z)$,
  and $\hat p(x):=\#x / N$ is the standard estimate for $P(x)$.
\end{theorem}

For the same size $\left|\mathbf{Z}\right|$,
the accuracy \eqref{eq:IID-front-2} that we have for the front-door criterion
appears much worse than the accuracy \eqref{eq:IID-back-2} for the back-door one.

\section{The adaptive setting with a fixed horizon}
\label{sec:core}

Under the \emph{strong interpretation} of Figure \ref{fig:back-repeated},
considered in this section,
each box $\text{past}_n$ stands for the whole past,
including the variables $X_i$, $Y_i$, and $Z_i$, $i\in[n]$.
Now each $X_{n+1}$, $n\in[N-1]$,
has incoming arrows (not shown explicitly in the figure)
from all variables at the previous steps,
including $X_i$, $Y_i$, and $Z_i$, $i\in[n]$.
As before, we allow repetition of any dag in Figure \ref{fig:back-repeated},
not just the one in Figure~\ref{fig:back}.

Our interpretation of Figure \ref{fig:back-repeated}
is that $X$ is a decision that has $Y$ as its result.
The decision at step $n+1$ may depend on the past decisions
and past values of $Y$, $Z$, and other variables.
In other words, the decision maker has access to all past observations.
In the case of Figure~\ref{fig:back},
all $Z_n$ are independent of the past and identically distributed;
all $Y_n$ have the same distribution given $X_n$ and $Z_n$,
and they are conditionally independent of the past.

For an integer $n\ge2$,
let $\llfloor n\rrfloor$ be the largest integer of the form $2^k$,
$k\in\{1,2,\dots\}$, satisfying $2^k\le n$
(and for $n<2$, $\llfloor n\rrfloor$ is defined as, say, 1).
We let $\lb$ stand for binary logarithm $\log_2$,
and we will often use it in the context of
$\lb\llfloor n\rrfloor=\lfloor\lb n\rfloor$
for a positive integer $n$.

It will be useful to extend the notation~\eqref{eq:number} and set, e.g.,
\begin{equation*}
  \#_m x y := \left|\{n\in[m]:(X_n,Y_n)=(x,y)\}\right|,
\end{equation*}
where $m$ may be different from $N$.
Earlier we defined standard estimates such as
$\hat p(y\mid\tilde x,z):=\#\tilde x y z / \#\tilde x z$
(and we will refrain from defining $\hat p$ in other similar contexts
in the following theorems).
We will also need the modification of $\hat p(y\mid\tilde x,z)$
defined by
\begin{equation}\label{eq:estimate}
  \dhat p(y\mid\tilde x,z)
  :=
  \frac
    {
      \left|\left\{
        n\in[N]:
        \#_n\tilde x z \le \llfloor\#\tilde x z\rrfloor,
        X_n=\tilde x, Y_n=y, Z_n=z
      \right\}\right|
    }
    {\llfloor\#\tilde x z\rrfloor}.
\end{equation}
In words, $\dhat p(y\mid\tilde x,z)$ is the fraction
of the first $\llfloor\#\tilde x z\rrfloor$ observations
with $X_n=\tilde x$ and $Z_n=z$ for which $Y_n=y$.
It is also an estimate of $P(y\mid\tilde x,z)$,
but it might not use all the available data
(however, it uses at least one half of the relevant observations).
We will also use the notation $\dhat p$ in other contexts,
such as $\dhat p(z\mid\tilde x)$.

Now we have to replace the Hoeffding inequality
used in our proof of Theorem~\ref{thm:IID-back} %
in Appendix~\ref{app:proofs-1} by the law of the iterated logarithm,
and so we will get our first iterated logarithm term.

\begin{theorem}\label{thm:core-back}
  Suppose the back-door criterion is satisfied.
  Fix $\delta>0$, $\tilde x$, and $y$;
  the time horizon $N$ is also fixed.
  Under the strong interpretation,  
  the following is a $(1-\delta)$-confidence interval
  for the parameter $P(y\mid\Do(\tilde x))$ defined by \eqref{eq:back}:
  the midpoint is
  \begin{equation}\label{eq:core-back-1}
    \sum_z
    \hat p(z)
    \dhat p(y\mid\tilde x,z)
  \end{equation}
  and the half-width is
  \begin{equation}\label{eq:core-back-2}
    \left|\mathbf{Z}\right|
    \sqrt{\frac{\ln\frac{4\left|\mathbf{Z}\right|}{\delta}}{2N}}
    +
    \sum_z
    \sqrt{
      \frac
        {2\ln\lb\llfloor\#\tilde x z\rrfloor+\ln\frac{6.6\left|\mathbf{Z}\right|}{\delta}}
        {2\llfloor\#\tilde x z\rrfloor}
    }.
  \end{equation}
\end{theorem}

\noindent
The case $\#\tilde x z\in\{0,1\}$ in \eqref{eq:core-back-2} requires special treatment;
namely, we set $\ln\lb1:=\infty$,
and so \eqref{eq:core-back-2} is interpreted as $\infty$ unless $\#\tilde x z\ge2$ for all $z$.

For the binary case of Figure~\ref{fig:back}, we can replace~\eqref{eq:core-back-2} by
\begin{equation}\label{eq:core-back-2-toy}
  2\sqrt{\frac{\ln\frac{6}{\delta}}{2N}}
  +
  \sum_{z\in\{0,1\}}
  \sqrt{\frac{2\ln\lb\llfloor\#\tilde x z\rrfloor+\ln\frac{10}{\delta}}{2\llfloor\#\tilde x z\rrfloor}},
\end{equation}
similarly to \eqref{eq:IID-back-2-toy}.

It is not clear how to extend Theorem~\ref{thm:core-back}
to the front-door criterion.
Moreover, it is not even obvious
that the quantity~\eqref{eq:front} is well-defined
under the strong interpretation,
since, e.g., the distribution of $X$ changes from step to step.
(For a demonstration that it is indeed well-defined
see, e.g., \cite[(3.27)]{Pearl:2009}.)

\section{The anytime-valid adaptive setting}
\label{sec:anytime}

Theorem~\ref{thm:core-back} can be easily extended to the setting
in which the time horizon $N$ is not fixed in advance.
Now we would like our results to be anytime valid,
with $N$ ranging over the positive integers $\{1,2,\dots\}$.
Since $N$ is variable,
now we will write $\dhat p_N(y\mid\tilde x,z)$
in place of $\dhat p(y\mid\tilde x,z)$ defined by \eqref{eq:estimate}
and add the lower index $N$ in other similar places;
in particular,
$\dhat p_N(z):=\#_{\llfloor N\rrfloor}z/\llfloor N\rrfloor$.
Remember that a \emph{$(1-\delta)$-confidence sequence}
is a sequence of (confidence) intervals
whose intersection covers the true parameter value
with probability at least $1-\delta$.

\begin{theorem}\label{thm:anytime-back}
  Suppose the back-door criterion is satisfied.
  Let $\delta>0$, $\tilde x\in\mathbf{X}$, and $y\in\mathbf{Y}$.
  The following is a $(1-\delta)$-confidence sequence
  for the parameter $P(y\mid\Do(\tilde x))$ defined by \eqref{eq:back}:
  the midpoint is
  \[
    \sum_z
    \dhat p_N(z)
    \dhat p_N(y\mid\tilde x,z)
  \]
  and the half-width is
  \begin{equation}\label{eq:anytime-back-2}
    \left|\mathbf{Z}\right|
    \sqrt{
      \frac
        {2\ln\lb\llfloor N\rrfloor+\ln\frac{6.6\left|\mathbf{Z}\right|}{\delta}}
        {2\llfloor N\rrfloor}
    }
    +
    \sum_z
    \sqrt{
      \frac
        {2\ln\lb\llfloor\#_N\tilde x z\rrfloor+\ln\frac{6.6\left|\mathbf{Z}\right|}{\delta}}
        {2\llfloor\#_N\tilde x z\rrfloor}
    }.
  \end{equation}
\end{theorem}

In the binary case of Figure \ref{fig:back},
\eqref{eq:anytime-back-2} can be slightly strengthened to
\begin{equation*}
  2
  \sqrt{
    \frac
      {2\ln\lb\llfloor N\rrfloor+\ln\frac{10}{\delta}}
      {2\llfloor N\rrfloor}
  }
  +
  \sum_{z\in\{0,1\}}
  \sqrt{
    \frac
      {2\ln\lb\llfloor\#_N\tilde x z\rrfloor+\ln\frac{10}{\delta}}
      {2\llfloor\#_N\tilde x z\rrfloor}
  }.
\end{equation*}

Theorem~\ref{thm:core-back}--\ref{thm:anytime-back} %
may be considered to be finite-sample analogues of Theorem~1 in \cite{Vovk:1996Gam}.
Their characteristic feature is the presence of iterated logarithm terms,
which are unavoidable (see Sect.~\ref{sec:necessity} for details)
and are especially prominent in the anytime-valid adaptive setting.

\begin{figure}
  \begin{center}
    \begin{tikzpicture} %
      \node (W) at (0,0) {W};
      \node (Z) at (2,0) {Z};
      \node (X) at (4,0) {X};
      \node (Y) at (6,0) {Y};
      \node (V) at (2,1) {V};
      \node (U) at (2,2) {U};
      \draw [->] (W) edge (Z);
      \draw [->] (Z) edge (X);
      \draw [->] (X) edge (Y);
      \draw [->] (V) edge (W);
      \draw [->] (V) edge (X);
      \draw [->] (U) edge (W);
      \draw [->] (U) edge (Y);
    \end{tikzpicture}
  \end{center}
  \caption{The napkin graph}
  \label{fig:napkin}
\end{figure}
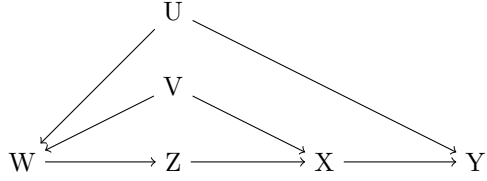

\begin{remark}
  In this paper we only discuss, outside of this remark,
  causal effects that are representable as arithmetic expressions
  involving only two arithmetic operations,
  plus and multiplication
  (minus could be added for free but is not useful).
  There are, however, situations in which the causal effect
  is given in a form involving division,
  such as the \emph{napkin graph}, shown in Figure~\ref{fig:napkin}.
  It has the following expression for the causal effect of $X$ on $Y$:
  \begin{equation}\label{eq:p-napkin}
    P(y\mid\Do(\tilde x))
    :=
    \frac
      {\sum_w P(y,\tilde x\mid z,w)P(w)}
      {\sum_w P(\tilde x\mid z,w)P(w)}
  \end{equation}
  (see \cite[Figure 2(c) and (1)]{Helske/etal:2021}).
  The expression~\eqref{eq:p-napkin} is a ratio,
  and our methods encounter even more serious difficulties
  than in the case of the front-door criterion.
  An interesting feature of the expression \eqref{eq:p-napkin}
  is that its right-hand side involves $z$ but does not really depend on it,
  as is clear from its left-hand side;
  this is an instance of so-called Verma constraints \cite{Bhattacharya/Nabi:2022}.
  The presence of $z$ on the right-hand side of \eqref{eq:p-napkin}
  is in a certain sense inevitable;
  formally, $Z$ is a ``trapdoor variable''
  as defined in \cite[Definition~3]{Helske/etal:2021}
  and explained in \cite[Sect.~2.3]{Helske/etal:2021}.
\end{remark}

\section{The necessity of the iterated logarithm term in the adaptive setting}
\label{sec:necessity}

The next theorem (proved in Appendix~\ref{app:proofs-2})
shows that an iterated logarithm term is unavoidable
already in the fixed-horizon setting of Theorem~\ref{thm:core-back}.
We assume, without loss of generality, $\left|\mathbf{Z}\right|=1$,
and so we ignore $\mathbf{Z}$.
Let us fix a confidence level $1-\delta$ (it can be arbitrarily close to 0),
$\tilde x\in\mathbf{X}$, and $y\in\mathbf{Y}$.
We are interested in estimating the causal effect
$P(y\mid\Do(\tilde x))=P(y\mid\tilde x)$.

For a positive integer $N$, we consider a confidence estimator $C_N$ for $P(y\mid\tilde x)$;
formally, $C_N$ maps a sequence $(x_1,y_1,\dots,x_N,y_N)\in(\mathbf{X}\times\mathbf{Y})^N$
to a (closed) subinterval $C_N(x_1,y_1,\dots,x_N,y_N)$ of $[0,1]$
and satisfies the following natural property of validity:
regardless of the underlying probability measure $P$,
\begin{equation}\label{eq:coverage}
  P(y\mid\tilde x)
  \in
  C_N(X_1,Y_1,\dots,X_N,Y_N)
\end{equation}
with probability at least $1-\delta$.
If $C$ is an interval, we let $\left|C\right|$ stand for its length.

\begin{theorem}\label{thm:necessity}
  Let $(C_N)_{N=1}^{\infty}$ be a family of confidence estimators
  and let $f:\{0,1,\dots\}\to(0,\infty)$ satisfy
  \[
    f(n)
    =
    o
    \left(
      \sqrt{\frac{\ln\ln n}{n}}
    \right),
    \qquad
    n\to\infty.
  \]
  If $N$ is sufficiently large,
  there exists $(x_1,y_1,\dots,x_N,y_N)\in(\mathbf{X}\times\mathbf{Y})^N$
  such that
  \begin{equation}\label{eq:goal}
    \left|
      C_N(x_1,y_1,\dots,x_N,y_N)
    \right|
    >
    f(k)
  \end{equation}
  where
  $
    k
    :=
    \sum_{i=1}^N
    1_{\{x_i=\tilde x\}}
  $
  is the number of occurrences of $\tilde x$.
\end{theorem}

\section{Applications to prediction sets}
\label{sec:prediction}

Theorems~\ref{thm:IID-back}--\ref{thm:anytime-back}
provide confidence intervals for causal effects,
whereas in \cite{Vovk/Wang:arXiv2603full} we were interested in prediction sets for $Y$.
In order to discuss connections
between our results here and the strong interpretation in \cite{Vovk/Wang:arXiv2603full},
in this section we will state a corollary of the toy version of Theorem~\ref{thm:core-back}
for the binary case of Figure~\ref{fig:back},
with~\eqref{eq:core-back-2} replaced by~\eqref{eq:core-back-2-toy},
giving prediction sets;
similar corollaries can be easily deduced
from Theorems~\ref{thm:IID-back}--\ref{thm:anytime-back} %
as well.
We consider the strong interpretation of Figure~\ref{fig:back},
as in Sect.~\ref{sec:core},
with $(X_n,Y_n,Z_n)$, $n\in[N]$,
complemented by another observation $Y$
with the probabilities of $Y=y$, $y\in\mathbf{Y}$,
given by the right-hand side of \eqref{eq:back}
for a fixed $\tilde x$.
Remember that $\mathbf{X}=\mathbf{Y}=\mathbf{Z}=\{0,1\}$.

\begin{corollary} %
  Fix $\delta>0$, $N$, and $\tilde x\in\mathbf{X}$.
  Then
  \begin{multline*} %
    \Gamma
    :=
    \Biggl\{
      y\in\mathbf{Y}:
      \sum_{z\in\{0,1\}}
      \hat p(z)
      \dhat p(y\mid\tilde x,z)
      +
      2\sqrt{\frac{\ln\frac{12}{\delta}}{2N}} \\
      +
      \sum_{z\in\{0,1\}}
      \sqrt{\frac{2\ln\lb\llfloor\#\tilde x z\rrfloor+\ln\frac{20}{\delta}}{2\llfloor\#\tilde x z\rrfloor}}
      >
      \frac{\delta}{2}
    \Biggr\}
  \end{multline*}
  is a $(1-\delta)$-prediction set for $Y$.
\end{corollary}

\begin{proof}
  We are required to prove that $Y\notin\Gamma$ with probability at most $\delta$.
  Let us fix $y\in\mathbf{Y}$ and prove
  that the probability of the conjunction of $Y=y$ and $y\notin\Gamma$
  is at most $\delta/2$.
  We will use the confidence interval
  $\eqref{eq:core-back-1}\pm\eqref{eq:core-back-2-toy}$
  with $\delta/2$ in place of $\delta$.
  Consider two cases:
  \begin{itemize}
  \item
    Suppose the right-hand-side of \eqref{eq:back} exceeds $\delta/2$.
    If $y\notin\Gamma$, the right end-point
    $\eqref{eq:core-back-1}+\eqref{eq:core-back-2-toy}$
    (with $\delta$ replaced by $\delta/2$)
    of the confidence interval is at most $\delta/2$,
    and the probability of this is at most $\delta/2$
    by the definition of a confidence interval.
  \item
    Otherwise, $Y=y$ with probability at most $\delta/2$
    by the definition of $Y$.
    \qedhere
  \end{itemize}
\end{proof}

This procedure is sub-optimal for several reasons.
One of them is that Hoeffding's inequality applied to the Bernoulli model
can be greatly improved when the probability of error is small;
see, e.g., Vapnik's \cite[Sects.~4.2 and~4.4]{Vapnik:1998} use
of multiplicative Chernoff inequalities
(in what he calls optimistic and pessimistic settings of learning problems).

The analogous corollary of Theorem~\ref{thm:IID-back} %
becomes more comparable with the results that we obtain in \cite{Vovk/Wang:arXiv2603full}.
However, the prediction sets derived in \cite{Vovk/Wang:arXiv2603full} are based on e-values
(are ``e-prediction sets'') whereas the prediction sets here are traditional ones.
We expect that methods of this paper lead to much looser results,
but their advantage is that they also work for the strong interpretation.

\section{Conclusion}
\label{sec:conclusion}

In this paper we derive confidence intervals and confidence sequences
for causal effects in IID and sequential non-IID settings.
These are some directions of further research:
\begin{itemize}
\item
  How do we establish analogues of Theorems~\ref{thm:core-back} and~\ref{thm:anytime-back}
  for the front-door criterion?
\item
  In this paper we concentrate on upper bounds for achievable widths
  of confidence intervals, and our results need to be complemented by lower bounds.
  Theorem~\ref{thm:necessity} is a first step in this direction.
\item
  This paper is based on the standard measure-theoretic probability
  \cite{Kolmogorov:1933,Shiryaev:2016,Shiryaev:2019}.
  To make our results as strong as possible,
  we could try and present them in the language of game-theoretic probability
  \cite{Shafer/Vovk:2019}
  (this was listed as direction of further research already in \cite[Sect.~5]{Vovk:1996Gam}).
\end{itemize}

\subsection*{Acknowledgments}

In literature search and brainstorming we used OpenAI Codex;
in particular, it noticed that Theorem~\ref{thm:necessity}
follows from the results of \cite{Duchi/Haque:2024}.
We take full responsibility for all statements made in this paper.

\appendix
\section{Proofs for Sects.~\ref{sec:IID}--\ref{sec:anytime}}
\label{app:proofs-1}

In the following two propositions
we consider an IID Bernoulli sequence $\xi_1,\xi_2,\dots$
with probability of success $p$
and use the notation $\hat p_n:=\frac1n\sum_{i=1}^n\xi_i$
for the standard estimate of $p$.
We start from a standard confidence interval for the probability of success
given by Hoeffding's inequality.

\begin{proposition}\label{prop:Okamoto}
  For a fixed $n$ and $\delta>0$,
  \begin{equation}\label{eq:interval}
    I
    :=
    \left\{
      p\in[0,1]:
      \left|
        p - \hat p_n
      \right|
      <
      \sqrt{\frac{\ln\frac{2}{\delta}}{2n}}
    \right\}
  \end{equation}
  is a $(1-\delta)$-confidence interval for $p$,
  in the sense of
  \begin{equation*}
    \P
    \left(
      p\in I
    \right)
    \ge
    1-\delta.
  \end{equation*}
\end{proposition}

\begin{proof}
  By Hoeffding's inequality
  (or Okamoto's earlier result \cite[Theorem~1]{Okamoto:1959}),
  for all $c>0$,
  \[
    \P
    \left(
      \left|
        p - \hat p_n
      \right|
      \ge
      c
    \right)
    \le
    2\exp(-2c^2n),
  \]
  which gives the confidence interval \eqref{eq:interval}.
\end{proof}

We will also need the following confidence sequence.

\begin{proposition}\label{prop:LIL}
  For each $\delta>0$,
  \begin{equation}\label{eq:sequence}
    I_n
    :=
    \left\{
      p:
      \left|
        p - \hat p_{\llfloor n\rrfloor}
      \right|
      <
      \sqrt{\frac{2\ln\lb\llfloor n\rrfloor+\ln\frac{3.3}{\delta}}{2\llfloor n\rrfloor}}
    \right\}
  \end{equation}
  is a $(1-\delta)$-confidence sequence,
  i.e.,
  \begin{equation}\label{eq:confidence}
    \P
    \left(
      \forall n:
      p\in I_n
    \right)
    \ge
    1-\delta.
  \end{equation}
\end{proposition}

\begin{proof}
  Similar confidence sequences can be obtained
  using Ville's \cite{Ville:1939} method of continuous mixtures of test martingales
  or its discrete analogue \cite[Sect.~5.1]{Shafer/Vovk:2019},
  but we will model our proof on \cite[Sect.~E]{Voracek:2024}.
  Fix $p\in[0,1]$.

  Since $\zeta(2)=\pi^2/6$,
  we can split the significance level $\delta$ into the series
  $\delta=\sum_{k=1}^{\infty}\delta_k$,
  where
  \[
    \delta_k
    =
    \frac{6}{\pi^2 k^2}
    \delta.
  \]
  Applying \eqref{eq:interval} with $n_k=2^k$ in place of $n$ and $\delta_k$ in place of $\delta$
  gives
  \begin{equation}\label{eq:last}
    \P
    \left(
      \left|
        p - \hat p_{n_k}
      \right|
      \ge
      \sqrt{\frac{\ln\frac{\pi^2k^2}{3\delta}}{2n_k}}
    \right)
    \le
    \delta_k.
  \end{equation}
  Finally, \eqref{eq:last} implies \eqref{eq:confidence} since
  $
    \pi^2/3
    \approx
    3.29 < 3.3
  $.
  (This argument works for $n\ge2$;
  otherwise,
  the inequality in \eqref{eq:confidence} is trivial
  since our convention,
  introduced in Sect.~\ref{sec:core}, is that
  $\ln\lb1:=\infty$.)
\end{proof}

Next we need a simple result from interval arithmetic.
We are only interested in subintervals of $[0,1]$.
Let $c\pm\Delta c$, where $c\in\R$ and $\Delta c\ge0$,
stand for the interval
\[
  c\pm\Delta c
  :=
  [c-\Delta c, c+\Delta c]
  \cap
  [0,1].
\]
For a binary operation $*$ on the reals (we are mostly interested in addition and multiplication),
we define its result on intervals pointwise:
\[
  I_1 \mathbin{*} I_2
  :=
  \{p_1\mathbin{*}p_2:
    p_1\in I_1, p_2\in I_2\}
  \cap
  [0,1].
\]

\begin{lemma} %
  For any two intervals $a\pm\Delta a$ and $b\pm\Delta b$,
  \begin{align}
    (a\pm\Delta a) + (b\pm\Delta b) &\subseteq (a+b)\pm(\Delta a+\Delta b), \label{eq:add}\\
    (a\pm\Delta a) \times (b\pm\Delta b) &\subseteq (a\times b)\pm(\Delta a+\Delta b). \label{eq:times}
  \end{align}
\end{lemma}

\begin{figure}[bt]
  \begin{center}
    \usetikzlibrary {patterns}
    \begin{tikzpicture}[scale=0.5]
      \fill[color=yellow!40!white] (4,0) -- (6,0) -- (6,8) -- (4,8);
      \fill[color=yellow!40!white] (0,3) -- (0,6) -- (8,6) -- (8,3);
      \fill[color=yellow] (4,3) -- (6,3) -- (6,6) -- (4,6);
      \draw (0,0) -- (8,0);
      \draw (0,3) -- (8,3);
      \draw (0,6) -- (8,6);
      \draw (0,8) -- (8,8);
      \draw (0,0) -- (0,8);
      \draw (4,0) -- (4,8);
      \draw (6,0) -- (6,8);
      \draw (8,0) -- (8,8);
      \filldraw [gray]
        (4,0) circle [radius=2pt]
        (6,0) circle [radius=2pt]
        (0,3) circle [radius=2pt]
        (0,6) circle [radius=2pt];
      \draw (4,0) node[anchor=north]{$a$} -- (6,0) node[anchor=north]{$a+\Delta a$};
      \draw (0,3) node[anchor=east]{$b$} -- (0,6) node[anchor=east]{$b+\Delta b$};
      \fill[pattern=dots] (4,3) -- (4,0) -- (6,0) -- (6,6) -- (0,6) -- (0,3);
    \end{tikzpicture}
  \end{center}
  \caption{Illustration of an inequality.}
  \label{fig:proof}
\end{figure}
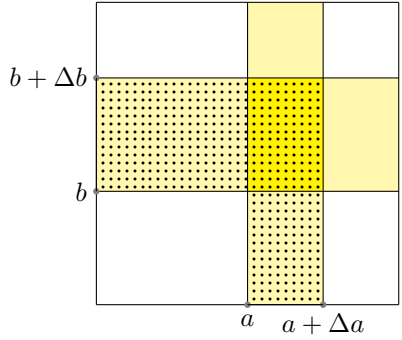

\begin{proof}
  The inclusion~\eqref{eq:add} is obvious, so we will only prove \eqref{eq:times}.
  The latter inclusion reduces to the conjunction of two inequalities:
  \begin{align}
    ((a-\Delta a)\vee0) ((b-\Delta b)\vee0) &\ge a b - (\Delta a+\Delta b), \label{eq:times-1} \\
    ((a+\Delta a)\wedge1) ((b+\Delta b)\wedge1) &\le a b + (\Delta a+\Delta b). \label{eq:times-2}
  \end{align}
  In the inequality~\eqref{eq:times-1} we can assume, without loss of generality,
  $a-\Delta a\ge0$ and $b-\Delta b\ge0$, which makes it obvious.
  In~\eqref{eq:times-2} we can assume, without loss of generality,
  $a+\Delta a\le1$ and $b+\Delta b\le1$, which reduces it to
  \begin{equation}\label{eq:final}
    (a+\Delta a)(b+\Delta b) - a b \le \Delta a+\Delta b.
  \end{equation}
  The last inequality is illustrated in Figure~\ref{fig:proof}:
  the dotted area of the plot represents the left-hand side of \eqref{eq:final},
  and the yellow area represents the right-hand side of \eqref{eq:final}
  (with the darker yellow area counted twice).
\end{proof}

\begin{figure}
  \begin{center}
    \begin{forest}
      [{$+$}
        [{$\times$}[{$P(y\mid\tilde x,0)$}][{$P(0)$}]]
        [{$\times$}[{$P(y\mid\tilde x,1)$}][{$P(1)$}]]
        [{$\times$}[{$P(y\mid\tilde x,2)$}][{$P(2)$}]]
      ]
    \end{forest}
  \end{center}
  \caption{The binary tree representing the polynomial expression \eqref{eq:back}
    over the formal variables $P(y\mid\tilde x,z)$ and $P(z)$
    in the case of $\mathbf{Z}=\{0,1,2\}$}
  \label{fig:poly}
\end{figure}
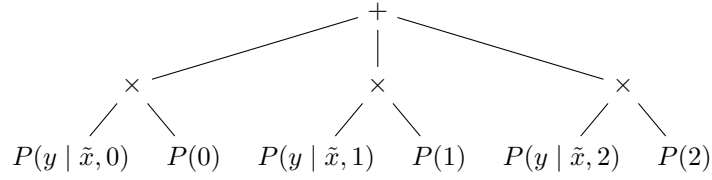

We will distinguish between (multivariate) polynomials and polynomial expressions.
A \emph{polynomial expression} is formed from a finite number of formal variables
by repeatedly applying the operations of multiplication and addition;
we do not allow constants (equivalently,
our polynomials and polynomial expressions are over the field $\{0,1\}$
and have a zero constant term).
A polynomial expression can be represented as a tree
(let us call it a \emph{polynomial tree}), not necessarily binary,
such as Figure \ref{fig:poly} in the case of the polynomial expression
given by the right-hand side of \eqref{eq:back} for $\mathbf{Z}=\{0,1,2\}$.
The same formal variable may be used several times.
A \emph{polynomial} is an equivalence class of polynomial expressions
where we do not distinguish polynomial expressions that reduce to each other
by applying the usual laws of commutativity, associativity, and distributivity
(associativity was already used implicitly
when we allowed non-binary multiplications and additions in our trees).
Without loss of generality we may assume that the operations
at different levels of polynomial trees alternate
(so at each level we have the same operation, ``$+$'' or ``$\times$'',
and perhaps some formal variables as leaf nodes;
the operations in adjacent levels are different).

\begin{corollary}\label{cor:combination}
  Let $E$ be a polynomial expression involving $m$ distinct formal variables.
  Define $E^*$ to be the polynomial obtained from $E$
  by replacing all multiplications by additions.
  Then, for any intervals $a_i\pm\Delta a_i$, $i=1,\dots,m$,
  \begin{equation}\label{eq:combination}
    E(a_1\pm\Delta a_1,\dots,a_m\pm\Delta a_m)
    \subseteq
    E(a_1,\dots,a_m)\pm E^*(\Delta a_i).
  \end{equation}
\end{corollary}

\noindent
The expression $E^*(\Delta a_i)$ in \eqref{eq:combination}
is, of course, the result of substituting $\Delta a_i$
for the formal variables in $E^*$ and evaluating the resulting expression.

\begin{proof}[Proof of Corollary~\ref{cor:combination}]
  We proceed by induction on the height of $E$ considered as polynomial tree.
  Repeatedly applying \eqref{eq:add}, we can extend it to any finite sums.
  Similarly, repeatedly applying \eqref{eq:times}, we can extend it to any finite products.
  This gives the statement \eqref{eq:combination} when the height of $E$ is 1.
  The inductive step is also provided by extensions of \eqref{eq:add} and \eqref{eq:times}
  to finite sums and products.
\end{proof}

\begin{proof}[Proof of Theorem~\ref{thm:IID-back}]
  In the proofs of Theorems~\ref{thm:IID-back}--\ref{thm:anytime-back} %
  we will use the slightly informal notation exemplified
  by $\eqref{eq:IID-back-1}\pm\eqref{eq:IID-back-2}$
  being the confidence interval with midpoint \eqref{eq:IID-back-1}
  and half-width \eqref{eq:IID-back-2}.

  We obtain the confidence interval $\eqref{eq:IID-back-1}\pm\eqref{eq:IID-back-2}$
  for \eqref{eq:back} (involving $2\left|\mathbf{Z}\right|$ probabilities)
  by combining the confidence intervals \eqref{eq:interval}
  for each of the $2\left|\mathbf{Z}\right|$ constituent probabilities.
  To ensure the overall confidence level $1-\delta$,
  we replace the $\delta$ in \eqref{eq:interval}
  by $\delta/(2\left|\mathbf{Z}\right|)$.
  By Corollary~\ref{cor:combination} applied to~\eqref{eq:back},
  we then indeed obtain the overall $(1-\delta)$-confidence interval
  with midpoint \eqref{eq:IID-back-1} and semi-width
  \begin{equation}\label{eq:IID-back-3}
    \sum_z
    \left(
      \sqrt{\frac{\ln\frac{4\left|\mathbf{Z}\right|}{\delta}}{2N}}
      +
      \sqrt{\frac{\ln\frac{4\left|\mathbf{Z}\right|}{\delta}}{2\#\tilde x z}}
    \right),
  \end{equation}
  i.e., \eqref{eq:IID-back-2}.
  The two square roots in \eqref{eq:IID-back-3} are the half-widths
  of the confidence intervals for $P(z)$ and $P(y\mid\tilde x,z)$, respectively,
  given by Proposition~\ref{prop:Okamoto}.
\end{proof}

To derive \eqref{eq:IID-back-2-toy} for Figure~\ref{fig:back} with binary variables,
notice that in the binary case we only need confidence intervals
for three constituent probabilities,
since an interval estimate for $P(Z=0)$
gives one for $P(Z=1)$ and vice versa.
This allows us to replace $\delta$ by $\delta/3$ rather than $\delta/4$
in \eqref{eq:interval}.

\begin{proof}[Proof of Theorem~\ref{thm:IID-front}]
  The proof is similar to that of Theorem~\ref{thm:IID-back}.
  We regard \eqref{eq:front} as a multivariate polynomial
  with formal variables $P(z\mid\tilde x)$ (indexed by $z\in\mathbf{Z}$),
  $P(y\mid x,z)$ (indexed by $(x,z)\in\mathbf{X}\times\mathbf{Z}$),
  and $P(x)$ (indexed by $x\in\mathbf{X}$).
  It is clear that \eqref{eq:IID-front-1} is the midpoint
  for the confidence interval given by Corollary~\ref{cor:combination},
  so we only needed to show that \eqref{eq:IID-front-2} is the resulting half-width.

  The total number of distinct formal variables in \eqref{eq:front}
  is $K:=\left|\mathbf{X}\right|\left|\mathbf{Z}\right|+\left|\mathbf{X}\right|+\left|\mathbf{Z}\right|$:
  \begin{itemize}
  \item
    there are $\left|\mathbf{Z}\right|$ of $P(z\mid\tilde x)$,
  \item
    there are $\left|\mathbf{X}\right|\left|\mathbf{Z}\right|$ of $P(y\mid x,z)$,
  \item
    and there are $\left|\mathbf{X}\right|$ of $P(x)$.
  \end{itemize}
  To ensure that \eqref{eq:interval} are simultaneous confidence intervals
  for all of them at confidence level $1-\delta$,
  we replace the $\delta$ in \eqref{eq:interval} by $\delta/K$.

  Corollary~\ref{cor:combination} applied to the polynomial expression \eqref{eq:front}
  gives the half-width
  \begin{multline}\label{eq:version-1}
    \sum_z
    \left(
      \sqrt{\frac{\ln\frac{2K}{\delta}}{2\#\tilde x}}
      +
      \sum_x
      \left(
        \sqrt{\frac{\ln\frac{2K}{\delta}}{2\#x z}}
        +
        \sqrt{\frac{\ln\frac{2K}{\delta}}{2N}}
      \right)
    \right)\\
    =
    \left|\mathbf{X}\right|\left|\mathbf{Z}\right|
    \sqrt{\frac{\ln\frac{2K}{\delta}}{2N}}
    +
    \left|\mathbf{Z}\right|
    \sqrt{\frac{\ln\frac{2K}{\delta}}{2\#\tilde x}}
    +
    \sum_{x,z}
    \sqrt{\frac{\ln\frac{2K}{\delta}}{2\#x z}},
  \end{multline}
  i.e., it gives us \eqref{eq:IID-front-2}.
  The square roots in \eqref{eq:version-1} are again coming
  from Proposition~\ref{prop:Okamoto}.
\end{proof}

\begin{remark}\label{rem:Horner}
  As far as the number of the operations ``$+$'' and ``$\times$'' is concerned,
  the expanded form of a polynomial
  is typically less efficient than what we get by applying multivariate Horner schemes
  (see, e.g., \cite{Ceberio/Kreinovich:2004});
  there are several other methods for optimizing the number of operations
  (see, e.g., \cite{Kuipers/etal:2013}).
  Namely, applying a multivariate Horner scheme leaves the same number of additions
  and reduces the number of multiplications.
  Since the right-hand side of \eqref{eq:front} is already in the Horner form
  obtained from the expanded form
  \begin{equation}\label{eq:expanded}
    P(y\mid\Do(\tilde x))
    =
    \sum_{x,z}
    P(z\mid\tilde x)
    P(y\mid x,z)
    P(x)
  \end{equation}
  of \eqref{eq:front} by starting from the formal variables $P(z\mid\tilde x)$,
  using it leads to a tighter confidence interval than using the expanded form
  (which we will spell out at the end of this remark).
  The multivariate Horner scheme depends on the order
  in which we apply it to different variables,
  and 
  \begin{equation}\label{eq:front-2}
    P(y\mid\Do(\tilde x))
    =
    \sum_x
    P(x)
    \sum_z
    P(y\mid x,z)
    P(z\mid\tilde x)
  \end{equation}
  is what we obtain in place of \eqref{eq:front}
  when we start from the formal variables $P(x)$.
  Using \eqref{eq:front-2} will give a different half-width
  from \eqref{eq:version-1}, namely
  \begin{multline}\label{eq:version-2}
    \sum_x
    \left(
      \sqrt{\frac{\ln\frac{2K}{\delta}}{2N}}
      +
      \sum_z
      \left(
        \sqrt{\frac{\ln\frac{2K}{\delta}}{2\#x z}}
        +
        \sqrt{\frac{\ln\frac{2K}{\delta}}{2\#\tilde x}}
      \right)
    \right)\\
    =
    \left|\mathbf{X}\right|
    \sqrt{\frac{\ln\frac{2K}{\delta}}{2N}}
    +
    \left|\mathbf{X}\right|\left|\mathbf{Z}\right|
    \sqrt{\frac{\ln\frac{2K}{\delta}}{2\#\tilde x}}
    +
    \sum_{x,z}
    \sqrt{\frac{\ln\frac{2K}{\delta}}{2\#x z}}.
  \end{multline}
  We cannot say \emph{a priori} which is larger,
  \eqref{eq:version-1} or \eqref{eq:version-2}.
  It is easy to check that \eqref{eq:version-1} is less than \eqref{eq:version-2}
  if and only if
  \begin{equation*}
    \frac{\#\tilde x}{N}
    <
    \left(
      \frac{1-1/\left|\mathbf{X}\right|}{1-1/\left|\mathbf{Z}\right|}
    \right)^2.
  \end{equation*}
  Therefore, the half-width \eqref{eq:version-1} looks likely to be better
  than~\eqref{eq:version-2} overall;
  in particular, \eqref{eq:version-1} is less than~\eqref{eq:version-2}
  when $\#\tilde x/N\le1/4$ or $\left|\mathbf{X}\right|\ge\left|\mathbf{Z}\right|$.

  The expanded form \eqref{eq:expanded} gives
  \begin{multline*}
    \sum_{x,z}
    \left(
      \sqrt{\frac{\ln\frac{2K}{\delta}}{2N}}
      +
      \sqrt{\frac{\ln\frac{2K}{\delta}}{2\#\tilde x}}
      +
      \sqrt{\frac{\ln\frac{2K}{\delta}}{2\#x z}}
    \right)\\
    =
    \left|\mathbf{X}\right|\left|\mathbf{Z}\right|
    \sqrt{\frac{\ln\frac{2K}{\delta}}{2N}}
    +
    \left|\mathbf{X}\right|\left|\mathbf{Z}\right|
    \sqrt{\frac{\ln\frac{2K}{\delta}}{2\#\tilde x}}
    +
    \sum_{x,z}
    \sqrt{\frac{\ln\frac{2K}{\delta}}{2\#x z}},
  \end{multline*}
  which is worse than both \eqref{eq:version-1} and \eqref{eq:version-2}.
\end{remark}

\begin{proof}[Proof of Theorem~\ref{thm:core-back}]
  We apply Proposition~\ref{prop:Okamoto} to estimating $P(z)$
  and Proposition~\ref{prop:LIL} to estimating $P(y\mid\tilde x,z)$ in \eqref{eq:back}.
  For Proposition~\ref{prop:Okamoto} to be applicable,
  the values of $Z_n$ at different steps in Figure~\ref{fig:back-repeated}
  should be IID, and this follows from no vertex in $Z_n$ being a descendant of $X_n$
  (one of the two conditions in the definition of the back-door criterion).
  For Proposition~\ref{prop:LIL} to be applicable,
  we need the values of $Y_n$ at the steps where $X_n=\tilde x$ and $Z_n=z$ to be IID,
  and this follows from $Y_n$ and the (possibly non-IID) variables at the previous steps
  being $d$-separated by $X_n\cup Z_n$.
  To check the $d$-separation \cite[Sect.~1.2.3]{Pearl:2009},
  notice that each path from a variable at the previous step to $Y_n$
  is blocked by $X_n$ if it has an arrow emanating from $X_n$
  and is blocked by $Z_n$ if it has an arrow entering $X_n$
  from a variable at step $n$.

  Since the total number of distinct formal variables in \eqref{eq:back}
  is $2\left|\mathbf{Z}\right|$,
  we replace the $\delta$ in \eqref{eq:interval} and \eqref{eq:sequence}
  by $\delta/(2\left|\mathbf{Z}\right|)$.
  Since \eqref{eq:core-back-1} is obviously the midpoint
  of the confidence interval given by Corollary~\ref{cor:combination},
  we only check that \eqref{eq:core-back-2} is its half-width.

  Plugging the half-width of the confidence interval for $P(z)$
  given by Proposition~\ref{prop:Okamoto}
  and the half-width of the confidence interval for $P(y\mid\tilde x,z)$
  given by Proposition~\ref{prop:LIL}
  into the polynomial expression \eqref{eq:back} with the multiplications replaced by additions,
  we obtain the overall half-width
  \begin{equation*}
    \sum_z
    \left(
      \sqrt{\frac{\ln\frac{4\left|\mathbf{Z}\right|}{\delta}}{2N}}
      +
      \sqrt{
        \frac
          {2\ln\lb\llfloor\#\tilde x z\rrfloor+\ln\frac{6.6\left|\mathbf{Z}\right|}{\delta}}
          {2\llfloor\#\tilde x z\rrfloor}
      }
    \right)
  \end{equation*}
  i.e., \eqref{eq:core-back-2}.
\end{proof}

In the case of Figure~\ref{fig:back} with binary variables,
we obtain \eqref{eq:core-back-2-toy}
if we again replace $\delta$ by $\delta/3$ rather than $\delta/4$
(and round up 9.9 to 10).

\begin{proof}[Proof of Theorem~\ref{thm:anytime-back}]
  The proof of Theorem~\ref{thm:anytime-back} is analogous,
  and the only difference is that we use the confidence sequence \eqref{eq:sequence}
  for estimating $P(z)$ as well.
\end{proof}

\section{Proofs for Sect.~\ref{sec:necessity}}
\label{app:proofs-2}

In this appendix we prove Theorem~\ref{thm:necessity}.
We assume, without loss of generality,
that $\mathbf{X}=\mathbf{Y}=\{0,1\}$,
$\tilde x=1$, and $y=1$.
The underlying probability measure $P$ defines the probability distribution of $X_n$
given the past $X_i,Y_i$, $i=1,\dots,n-1$,
and also the probabilities $p_1:=P(Y_n=1\mid X_n=1)$ and $p_0:=P(Y_n=1\mid X_n=0)$.

Theorem~\ref{thm:necessity} will be deduced from the following special case
of Proposition~11 in \cite{Duchi/Haque:2024}.

\begin{proposition}\label{prop:DH}
  Let $(I_n)$ be a $(1-\delta)$-confidence sequence for the Bernoulli model.
  Then, for infinitely many $n$, there exists $(y_1,\dots,y_n)\in\{0,1\}^n$
  such that $\left|I_n(y_1,\dots,y_n)\right|>f(n)$.
\end{proposition}

\begin{proof}
  It suffices to prove the statement of the proposition
  for the Bernoulli submodel in which the probability of success $p$
  is restricted to $p\in(1/4,3/4)$.
  We just need to check the conditions of \cite[Proposition 11]{Duchi/Haque:2024}.
  The key condition is that
  \[
    \KL(B_{p},B_{p'})
    =
    O((p-p')^2),
  \]
  where $\KL$ stands for Kullback--Leibler distance,
  and it follows from Taylor's formula:
  \[
    \KL(B_{p},B_{p+x})
    =
    -p
    \ln
    \left(
      1+\frac{x}{p}
    \right)
    -(1-p)
    \ln
    \left(
      1-\frac{x}{1-p}
    \right)
    =
    \frac{x^2}{2p(1-p)}
    +
    O(x^3),
  \]
  where $O$ is uniform in $p$.
\end{proof}

To prove Theorem~\ref{thm:necessity}, we argue indirectly.
Suppose such a family $(C_N)$ of confidence estimators exists.
First we turn it into a family of finite confidence sequences
for the Bernoulli model.
Define the ``padded intervals''
\[
  I_n^N(y_1,\dots,y_n)
  :=
  C_N(1,y_1,\dots,1,y_n,0,0,\dots,0,0),
  \quad
  n=0,\dots,N
\]
(we apply $C_N$ to the $n$ pairs $(1,y_i)$, $i=1,\dots,n$,
followed by $N-n$ pairs $(0,0)$).

\begin{lemma}
  For each $N$,
  the padded intervals form a confidence sequence of length $N$
  for the Bernoulli model:
  for each $p\in(0,1)$,
  $\P(p\in\cap_n I^N_n(\xi_1,\dots,\xi_n))\ge1-\delta$,
  where $\xi_1,\dots,\xi_N$ are IID and Bernoulli
  with probability of success $p$.
\end{lemma}

\begin{proof}
  Given $p$, define the underlying probability measure as follows:
  let $p_1:=p$ and $p_0:=0$;
  set $X_n:=1$, $n=1,2,\dots$, and generate $Y_n$
  (setting $Y_n:=1$ with probability $p$)
  until
  \begin{equation}\label{eq:stopping}
    p
    \notin
    I_n^N(Y_1,\dots,Y_n).
  \end{equation}
  As soon as \eqref{eq:stopping} happens,
  start setting $X_n:=0$ (and, therefore, $Y_n:=0$ a.s.).
  It remains to notice that \eqref{eq:stopping} means that \eqref{eq:coverage} is violated.
\end{proof}

Suppose \eqref{eq:goal} is violated for all $(x_1,y_1,\dots,x_N,y_N)$
for arbitrarily large $N$;
fix a sequence $N_j$, $j=1,2,\dots$, of such successful horizons $N$
(they are successful for the confidence estimator:
it achieves accuracy $f$).
Therefore, we always have
\begin{equation*} %
  \left|
    C_{N_j}(x_1,y_1,\dots,x_{N_j},y_{N_j})
  \right|
  \le
  f(x_1+\dots+x_{N_j}).
\end{equation*}
In particular, we have
\begin{equation}\label{eq:I}
  \left|
    I_n^{N_j}(y_1,\dots,y_{n})
  \right|
  \le
  f(n)
\end{equation}
for all padded intervals.

Now let us pass to a limit along subsequences.
The set of all binary strings $y_1,\dots,y_n$ is countable,
and the closed subintervals of $[0,1]$ form a compact set
in the natural topology
(where convergence is defined as the convergence of both end-points).
By a diagonal compactness argument,
there is a subsequence, which we still denote $N_j$,
such that for every $n$ and every $y_1,\dots,y_n$,
\[
  I_n^{N_j}(y_1,\dots,y_n) \to I_n(y_1,\dots,y_n)
\]
for some limit interval $I_n(y_1,\dots,y_n)$.
Namely, we find a subsequence of $N_j$, denoted $N'_j$,
such that $I_n^{N'_j}(y_1,\dots,y_n)$ converges for the first $y_1,\dots,y_n$,
then we find a subsequence $N''_j$ of $N'_j$
such that $I_n^{N''_j}(y_1,\dots,y_n)$ converges for the second $y_1,\dots,y_n$,
etc.;
the resulting subsequence $N_j$ is formed by taking the first element of $N'_j$,
the second element of $N''_j$, etc.

Since $I^{N_j}_n$ form finite confidence sequences,
$I_n$ will form a confidence sequence for the Bernoulli model.
By \eqref{eq:I}, we always have
$
  \left|
    I_n
  \right|
  \le
  f(n)
$.
This contradicts Proposition~\ref{prop:DH}.
\end{document}